\documentclass[11pt,letterpaper]{amsart}

\usepackage{tikz}
\tikzstyle{every node}=[circle, draw, fill=black!50,
                        inner sep=0pt, minimum width=3pt]

\usepackage{amsfonts,amsmath,latexsym,color,epsfig,hyperref,enumitem, amssymb,bbm, amsthm, mathrsfs}
\usepackage{verbatim}
\usepackage{macro}
\usepackage{asymptote}

\renewcommand{\leq}{\leqslant}
\renewcommand{\geq}{\geqslant}

\newcommand{\sF}{\mathscr{F}}

\def\qed{\ifvmode\mbox{ }\else\unskip\fi\hskip 1em plus 10fill$\Box$}

\input{epsf}

\makeatletter
\def\Ddots{\mathinner{\mkern1mu\raise\p@
\vbox{\kern7\p@\hbox{.}}\mkern2mu
\raise4\p@\hbox{.}\mkern2mu\raise7\p@\hbox{.}\mkern1mu}}
\makeatother

\def\R{\mathbb R}

\def\C{\mathbb C}
\def\F{\mathbb F}

\newtheorem{innercustomgeneric}{\customgenericname}
\providecommand{\customgenericname}{}
\newcommand{\newcustomtheorem}[2]{%
  \newenvironment{#1}[1]
  {%
   \renewcommand\customgenericname{#2}%
   \renewcommand\theinnercustomgeneric{##1}%
   \innercustomgeneric
  }
  {\endinnercustomgeneric}
}

\newcustomtheorem{customthm}{Conjecture}
\newcustomtheorem{customlemma}{Lemma}

\title{Colorful Helly via induced matchings} 

\author{Cosmin Pohoata}
\thanks{Department of Mathematics, Emory University, Atlanta, GA. Email: {\tt cosmin.pohoata@emory.edu}. Research supported by NSF Award DMS-2246659.}

\author{Kevin Yang}
\thanks{Skyline High School, Sammamish, WA. Email: {\tt yangkev27@issaquah.wednet.edu}.}

\author{Shengtong Zhang}
\thanks{Department of Mathematics, Stanford University, Stanford, CA. Email: {\tt stzh1555@stanford.edu}.}

\begin{document}
\maketitle

\begin{abstract}
     We establish a theorem regarding the maximum size of an {\it{induced}} matching in the bipartite complement of the incidence graph of a set system $(X,\sF)$. We show that this quantity plus one provides an upper bound on the colorful Helly number of this set system, i.e. the minimum positive integer $N$ for which the following statement holds: if finite subfamilies $\sF_1,\ldots, \sF_{N} \subset \sF$  are such that $\cap_{F \in \sF_{i}} F = 0$ for every $i=1,\ldots,N$, then there exists $F_i \in \sF_i$ such that $F_1 \cap \ldots \cap F_{N} = \emptyset$. We will also discuss some natural refinements of this result and applications.
\end{abstract}

\section{Introduction}


In 1913, Helly famously proved that for every collection of convex sets $C_1,\ldots,C_n$ in $\mathbb{R}^{d}$ with an empty intersect, there exists a set $S \subset [n]$ of size $|S| \leq d+1$ with $\bigcap_{i \in S} C_i = \emptyset$. Moreover, $d+1$ is the smallest number with this property. This result is known as Helly's theorem and is fundamental result in combinatorial geometry, with a large number of generalizations and applications over the years. See for example \cite{Ba}, \cite{DGK}, \cite{Eckhoff}, \cite{Matousek}, and the references therein. We also recommend the surveys \cite{ALS} and \cite{BK}, focusing on more recent developments. To state some of these different variants of Helly's theorem, it is often standard to introduce the following terminology. For a set $X$ and a family $\sF$ of subsets of $X$, the {\it{Helly number}} $h = h(X,\sF) \in \mathbb{N}$ of $\sF$ is defined as the smallest integer $h$ such that for any finite subfamily $K$ of $\sF$, if every subset of at most $h$ members of $\sF$ has a nonempty intersection then all sets in $K$ have a nonempty intersection. If no such $h$ exists, then $h(X,\sF) := \infty$. Helly's theorem above then asserts that the Helly number of the family of convex sets in $\mathbb{R}^{d}$ is $d+1$. Another example of a known Helly number is the Helly number of spheres in $\mathbb{R}^{d}$. In \cite{Maehara}, Maehara proved that this equals $d+2$. An additional example is the theorem of Doignon \cite{Doignon} which asserts that the Helly number of convex lattice sets in $\mathbb{R}^{d}$, that is sets of the form $C \cap \mathbb{Z}^{d}$ where $C$ is a convex set in $\mathbb{R}^{d}$ is $2^{d}$. A combinatorial example is
the fact that the Helly number of the collection of (sets of vertices of) subtrees of any given tree is $2$. 

An algebraic example that in some sense motivated the present paper is the fact that for any field $\mathbb{F}$ the Helly number of hypersurfaces of degree at most $D$ in $\mathbb{F}^{d}$ is ${D+d \choose d}$. This was proved in 1986 by Deza and Frankl \cite{DezaFrankl} in the following elegant way: for any set of polynomials $f_1,\ldots,f_n$ of degree $\leq D$ in $\mathbb{F}[x_1,\ldots,x_d]$ with the property that their zero sets $Z(f_i) := \left\{ x \in \mathbb{F}^{d}: f_i(x) = 0\right\}$ satisfy $\bigcap_{i=1}^{n} Z(f_i) = \emptyset$, consider a {\it{minimal}} subset $S \subset [n]$ such that $\bigcap_{i \in S} Z(f_i) = \emptyset$. By definition, this is a set $S$ which must also satisfy the property that for every $j \in S$, the intersection $\bigcap_{i \in S \setminus \{j\}} Z(f_i)$ is nonempty. In other words, for every $j \in S$ there exists $x_j \in \mathbb{F}^{d}$ such that: 
\begin{equation} \label{matching}
    f_{i}(x_j) \neq 0\ \text{if and only if}\ i \neq j.
\end{equation}
By a (nowadays) standard argument, it is not difficult to see that this condition implies that the set of polynomials $\left\{f_{i}\right\}_{i \in S}$ is a set of linearly independent polynomials in $\mathbb{F}[x_1,\ldots,x_d]$. This yields the estimate $|S| \leq \binom{D+d}{d}$, since $\binom{D+d}{d}$ equals precisely the dimension of the vector space of polynomials of degree at most $D$ in $\mathbb{F}[x_1,\ldots,x_d]$. In \cite{Frankl}, Frankl subsequently extended this argument to varieties of bounded degree. It is perhaps instructive to highlight at this point that a similar strategy cannot be used to prove Helly's original theorem for convex sets. In particular, if $d \geq 2$ and $C_1,\ldots,C_{n}$ is a collection of convex sets in $\mathbb{R}^{d}$ such that $\bigcap_{i=1}^{n} C_{i} =\emptyset$, then it is no longer true that the size of a {\it{minimal}} set $S \subset [n]$ with the property that $\bigcap_{i \in S} C_i = \emptyset$ has to be bounded independently of $n$ (rather than a {\it{minimum}} set). Indeed, note that while for every $j \in S$ there exists a point $x_j \in \mathbb{R}^{d}$ such that
\begin{equation} \label{bad}
    x_j \not\in C_i\ \text{if and only if}\ i \neq j,
\end{equation}
this condition no longer implies that $|S|=O_{d}(1)$ (let alone $|S| \leq d+1$).\footnote{Indeed, note that for any $d \geq 2$ and any arbitrarily large $m \geq 1$ one can construct $m$ convex sets $C_1,\,\ldots,\,C_m\in\mathbb{R}^d$ such that for every $i\in[m]$ there exists $x_i$ such that $x_i \in \bigcap_{j\neq i}C_j\setminus C_i$. 
Let $P$ be a convex polytope with at least $m$ facets. Pick $m$ facets $f_1,\,\ldots,\,f_m$ and let $C_i$ be the half-space bounded by the hyperplane supporting $f_i$ and containing $P$. Let $x_i$ be any point in the polytope determined by the hyperplane supporting $f_i$, and the hyperplane supporting all the neighbor facets of $f_i$. By construction, $x_i\in C_j$ if and only if $i\neq j$.} 
Nevertheless, this type of argument turns out to be quite powerful in settings where conditions like \eqref{matching} or \eqref{bad} do generate a set of $|S|$ linear independent objects in a vector space of bounded dimension. For example, in a recent nice paper \cite{AJS}, Alon, Jin, and Sudakov also used this strategy to determine that the Helly number of the set of Hamming balls of radius $t$ inside an $n$-dimensional space $X^{n}$ is $2^{t+1}$ (where $X$ is an arbitrary finite or infinite set). 

In this paper, we generalize this story by showing that in a set system exhibiting conditions like \eqref{matching} or \eqref{bad} one can also establish an optimal colorful Helly theorem. To state things more precisely going forward, it will be convenient to consider the following abstract setting. Let $X$ be a set and let $\sF$ be a collection of subsets of elements from $X$. We define the colorful Helly number of the set system $(X,\sF)$ to be the minimum positive integer $\eta=\eta(X,\,\sF)$ for which the following statement holds: if finite subfamilies $\sF_1,\,\ldots,\, \sF_{\eta} \subset \sF$ are such that $\cap_{F \in \sF_{i}} F = \emptyset$ for every $i=1,\,\ldots,\,\eta$, then there exists $F_i \in \sF_i$ such that $F_1 \cap \cdots \cap F_{\eta} = \emptyset$. 



Like before, if no such $\eta$ exists, then we say $\eta(X,\,\sF):=\infty$. It is a celebrated theorem of Lov\'asz (whose proof first appeared in a paper by B\'ar\'any \cite{Ba82}) that if $X = \mathbb{R}^{d}$ and $\sF$ is the family of convex sets in $\mathbb{R}^{d}$ then $\eta(X,\,\sF) = d+1$. This generalization of Helly's theorem has managed to inspire a lot of further interesting research directions on its own, due to its rather mysterious geometric nature. For example, while clearly $h(X,\sF) \leq \eta(X,\sF)$, it is not always the case that $h(X,\sF) = \eta(X,\sF)$. Let $X = \mathbb{R}^{d}$ and let $\sF$ consist of a family of axis-parallel boxes: it is a folklore result that in this case $h(X,\sF)=2$, while it is not difficult to see that $\eta(X,\sF)=d+1$.\footnote{Lov\'asz's theorem implies $\eta(X,\sF) \leq d+1$, whereas $\eta(X,\sF)\geq d+1$ follows by considering the $d$-coloring of the facets of a hypercube in $\mathbb{R}^{d}$ so that each opposite pair of facets carry the same color. This is a coloring with $d$ color classes $\sF_{1},\ldots,\sF_{d}$ where $\cap_{F \in \sF_{i}} F = 0$ for every $i=1,\,\ldots,\,\eta$, but such that  $F_1 \cap \cdots \cap F_{d} \neq \emptyset$ for every $F_i \in \sF_i$.} It is therefore natural to ask which Helly-type theorems admit appropriate colorful versions (see for example the survey of Amenta, De Loera, and Sober\'on \cite[Section 2.2]{ALS} for more context). 

Our results below address this question by providing several natural sufficient conditions for the existence of a colorful Helly theorem. We will work with the incidence graph of $X$ and $\sF$, which is the bipartite graph $G_{X,\,\sF}$ on $X \cup \sF$ where $(x,\,F) \in X \times \sF$ is an edge if $x \in F$. The bipartite complement $\overline{G}_{X,\,\sF}$ of $G_{X,\,\sF}$ denotes the graph with vertex set $X \cup \sF$ where $(x,\,F) \in X \times \sF$ is an edge if $x \not \in F$. We will also be concerned with the size $\tau=\tau(X,\,\sF)$ of the maximum induced matching inside $\overline{G}_{X,\,\sF}$. We shall sometimes refer to a matching in $\overline{G}_{X,\,\sF}$ as a {\it{comatching}} of the set system $(X,\sF)$, and to $\tau(X,\,\sF)$ as the size of the largest comatching in $(X,\sF)$ (or simply, the comatching number of $(X,\sF)$). We are now ready to state our first result.

 \begin{theorem} \label{main}
Let $X$ be a set and let $\sF$ be a collection of subsets of $X$. Then we have
$$\eta(X,\,\sF) \leq 1+ \tau(X,\,\sF).$$
 \end{theorem}


As we will see in Section 2, this result is optimal and even the additive $+1$ term is necessary. Furthermore, in certain situations we can remove the assumption that the $\sF_i$'s are finite. Let $\eta^\infty(X, \sF)$ be the colorful Helly number with the finiteness assumption on $\sF_i$ removed. Clearly, we have $\eta(X, \sF) \leq \eta^\infty(X, \sF)$.

We say $\sF$ is \textit{Noetherian} if for any sequence of sets $F_1, F_2,\ldots$ in $\sF$, the descending chain of sets
$$F_1 \supset F_1 \cap F_2 \supset F_1 \cap F_2 \cap F_3 \supset \cdots$$
stabilizes. For example, we have
\begin{itemize}
    \item If $X$ is finite, then $\sF$ is Noetherian.
    \item If $X$ is $\C^n$ or $\R^n$, and $\sF$ is any family of algebraic varieties in $X$, then $\sF$ is Noetherian due to the Nullstellensatz and the fact that $\C[x_1,\ldots, x_n]$ is a Noetherian ring.
    \item If $X$ is $\R^2$, and $\sF$ is the family of convex sets, then $\sF$ not Noetherian. We can check that $\eta(X, \sF) = 3$ but $\eta^\infty(X, \sF) = \infty$.
\end{itemize}
If $\sF$ is Noetherian, then the finiteness assumption on $\sF_i$ is not necessary. While this observation is not hard to prove, we are unaware of any literature that discusses this subtlety. 
\begin{proposition} \label{noetherian}
    If a set family $\sF$ on $X$ is Noetherian, then we have
    $$\eta^\infty(X, \sF) = \eta(X, \sF).$$
\end{proposition}
In addition, we can consider a slightly stronger notion of $\tau(X, \sF)$. Let $\tau'(X, \sF)$ be the largest $\tau'$ such that there exist $x_1,\ldots, x_{\tau' + 1} \in X$ and $F_1,\ldots, F_{\tau'}$ such that $x_i \in F_j$ if and only if $i \neq j$. We call this a \textit{comatching with intersection}, and $x_{\tau' + 1}$ the \textit{common point.} Clearly, we have
$$\tau'(X, \sF) \in \{\tau(X, \sF) - 1, \tau(X, \sF) \}.$$
Using this terminology, we will show the following refinement of Theorem \ref{main}, which is in some sense our main result.

\begin{theorem} \label{main'}
For any set family $\sF$ on $X$, we have
$$\eta(X, \sF) \leq 1 + \tau'(X, \sF).$$
In particular, if $\tau'(X, \sF) = \tau(X, \sF) - 1$, then $\eta(X, \sF) = 1 + \tau'(X, \sF) = \tau(X, \sF)$.
\end{theorem}
Theorem \ref{main'} has a few new consequences, which we will discuss in detail in Section 3. First, we will argue that Theorem \ref{main'} implies that the colorful Helly number for $d$-dimensional spheres in $\mathbb{R}^{d}$ is $d+2$, generalizing the work of Maehara \cite{Maehara}. Second, we will show that Theorem \ref{main'} establishes an optimal colorful Helly theorem for any set of hypersurfaces of bounded degree, generalizing the works of Motzkin \cite{Motzkin} and Deza-Frankl \cite{DezaFrankl}. Last but not least, we will show that the colorful Helly number of the set of Hamming balls of radius $t$ inside an $n$-dimensional space $X^{n}$ is $2^{t+1}$, generalizing the work of Alon, Jin, and Sudakov \cite{AJS}. Turns out, it is also possible to prove the first two consequences of Theorem \ref{main'} using some more ad-hoc arguments. These arguments have some shortcomings, but we considered them to be quite instructive so we will discuss them in Section 2.

It is also natural to wonder whether set systems with bounded comatching number satisfy a so-called fractional Helly theorem. In the context of Helly’s theorem for convex sets, this theorem was first established by Katchalski and Liu in \cite{KL}: if $\sF$ is a family of $n \geq d + 1$ convex sets in $\mathbb{R}^{d}$ such that the number of $(d + 1)$-tuples of $\sF$ with non-empty intersection is
at least $\alpha {n \choose d+1}$ for some constant $\alpha > 0$, then there are at least $\beta n$ members of $\sF$ whose intersection is non-empty, where $\beta > 0$ is a constant which depends only
on $\alpha$ and $d$. The optimal choice $\beta=1-(1-\alpha)^{\frac{1}{d+1}}$ was determined by Kalai \cite{Kalai}. In \cite{Holmsen}, Holmsen showed that the the fractional Helly theorem can be derived
as a purely combinatorial consequence of the colorful Helly theorem, and then later Holmsen and Lee \cite{HolmsenLee} used this result to prove that a fractional Helly theorem holds in all set systems with a bounded Radon number. We note that an abstract fractional Helly theorem for sets with bounded comatching number also holds. 

\begin{corollary} \label{frac}
Let $(X,\,\sF)$ be a system with finite comatching number $\tau$. For $\alpha\in (0,\,1)$ there exists $\beta=\beta(\alpha,\,\tau)\in (0,\,1)$ with the following property: Let $F$ be a subfamily of $n>\tau+1$ sets in $\sF$ and suppose at least $\alpha\binom{n}{\tau+1}$ of the $(\tau+1)$-tuples in $F$ have nonempty intersection. Then there exists some $\beta n$ members of $F$ whose intersection is non-empty.
\end{corollary}

While it is also possible to use the result \cite{Holmsen} and derive a fractional Helly theorem using Theorem \ref{main'}, we note that the statement from Corollary \ref{frac} (with the correct dependence on $\tau$) follows from the work of Matou\v sek \cite{Matousek2}. Matou\v sek showed that the fractional Helly number of a set system is at most the dual VC dimension plus one. On the other hand, the dual VC-dimension of a set system is dominated by our parameter $\tau'$: Suppose we are given a family $A\subset\sF$ which is dual shattered by $X$, meaning, for every $B\subset A$, there exists $x\in\bigcap_{b\in B} B\setminus\bigcap_{a\in A} A$. Taking $B:=A\setminus\{a\}$ for $a\in A$ and $B:=A$ yields that $A$ is a comatching with intersection. 

\section*{Ackowledgement}
We thank Qiuyu Ren for numerous helpful discussions, and proving the $d = 2$ case of Proposition~\ref{prop:no-leray}.


\bigskip

\section{Warm up: Colorful Helly for hypersurfaces}

In this section, we discuss some preliminary results that only require ad-hoc arguments (in hindsight). The first one concerns a colorful version of Helly's theorem for spheres. 

\subsection{Spheres}
In \cite{Maehara}, Maehara proved that the Helly number for $d$-dimensional spheres is $d+2$. We show that the colorful Helly number is equal to the Helly number. We note that we also do not need the finiteness assumption on the families $\sF_i$.

\begin{theorem}[Colorful Helly for spheres] \label{spheres}
Let $\sF_1,\,\ldots,\,\sF_{d+2}$ be families of spheres in $\mathbb{R}^d$. If for every transversal $\{S_i\in\sF_i: 1\leq i\leq d+2\}$, the intersection $\bigcap_{i=1}^{d+2} S_i$ is nonempty, then one of the $\sF_i$ has nonempty intersection.
\end{theorem}

\begin{proof}
We use the fact that given a $d$-dimensional sphere $A$ and a $(d-1)$-dimensional sphere $B$, they are either disjoint, tangent, $A$ contains $B$, or $A\cap B$ is a $(d-2)$-dimensional sphere.

Let $S_1$ and $S_2$ be spheres in $\sF_1$ and $\sF_2$. Assume they are not tangent, as otherwise every $S_3\in\sF_3$ must contain this tangent point, so we are done. Then $S_1$ and $S_2$ intersect in a $(d-1)$-dimensional sphere $C_{d-1}$. If every $S_3\in\sF_3$ contains $C_{d-1}$, $\sF_3$ has a nonempty intersection and we are done. If there exists $S_3\in\sF_3$ that is tangent to $C_{d-1}$, every $S_4\in\sF_4$ must contain this tangent point, so we are done. Hence we may assume there exists $S_3\in\sF_3$ that intersects $C_{d-1}$ in a $(d-2)$-dimensional sphere $C_{d-2}$. Continuing like this, we get spheres $S_i\in\sF_i$ for $1\leq i\leq d+1$ that intersect in a point ($0$-dimensional sphere) $C_0$. Then all $S_{d+2}\in\sF_{d+2}$ must contain $C_0$ so $\sF_{d+2}$ has a nonempty intersection.
\end{proof}

We would like to point out that Theorem \ref{spheres} is an example where Theorem \ref{main} by itself only yields a suboptimal result, as in this case $\eta = \tau = d+2$. On the other hand, we shall see in Section 3 that Theorem \ref{main'} does provide the right answer.

\begin{proposition}\label{spheres}

Let $X:=\mathbb{R}^d$ and $\mathcal{S}_d\subset\mathbb{R}^d$ be the family of all $d$-dimensional spheres. Then 
$$\tau(X,\,\mathcal{S}_d) = d+2.$$

\end{proposition}

\begin{proof}
First, we show that $\tau(X,\,\mathcal{S}_d)\leq d+2$ through a stronger statement: for positive integers $d\leq n$, $X:=\mathbb{R}^n$, and the family $\mathcal{S}_d\subset\mathbb{R}^d$ of all $d$-dimensional spheres, we have $\tau(X,\,\mathcal{S}_d)\leq d+2$. 

We proceed with induction on $d$. For $d=0$, a sphere is a point. Clearly $\tau(X,\,\mathcal{S}_0)=2$.


Assume the statement for $d-1$. Assume for the sake of contradiction there exists $d+3$ spheres $S_1,\,\ldots,\,S_{d+3}\in\mathcal{S}_d$ and $d+3$ (distinct) points $x_1,\,\ldots,\,x_{d+3}$ such that $x_i\in S_j$ if and only if $i\neq j$. Consider the set $Y=\{x_1,\,\ldots,\,x_{d+2}\}$ of points on $S_{d+3}$. Then $S_i$ contains $Y\setminus\{x_i\}$ for all $i$. Let $S_i':=S_i\cap S_{d+3}$ for $1\leq i\leq d+2$, which is a $(d-1)$-dimensional sphere. Then $x_i\in S_j'$ if and only if $i\neq j$, so this forms an induced matching in $\overline{G}_{X,\,\mathcal{S}_{d-1}}$. By the induction hypothesis, $d+2\leq (d-1)+2$, a contradiction.

To see that $\tau(X,\,\mathcal{S}_d) \geq d+2$ and thus that $\tau(X,\,\mathcal{S}_d) =d+2$, consider $T$ to be a regular $d$-dimensional simplex. Let $\sF$ be the family of spheres consisting of each sphere centered at the vertices of $T$ and containing the center of $T$, and the sphere circumscribing $T$. It is easy to see that $\tau(X,\,\sF)=d+2$, and thus $\tau(X,\,\mathcal{S}_d) \geq d+2$.

The construction for $d=2$ is shown below.



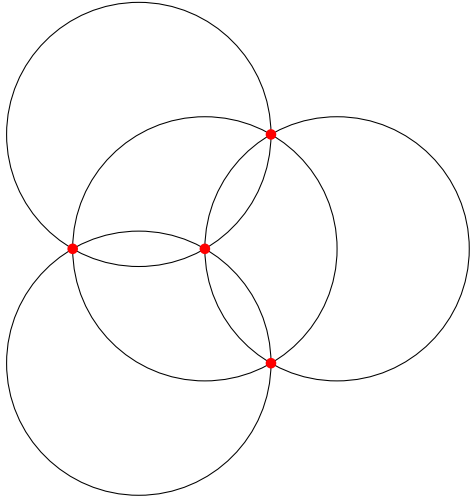
\begin{figure}[h]

\begin{tikzpicture}[x=1pt,y=1pt]

\draw (50,0) circle (50);
\draw ({50*cos(120)},{50*sin(120)}) circle (50);
\draw ({50*cos(240)},{50*sin(240)}) circle (50);
\draw (0,0) circle (50);

\fill[red] (0,0) circle (2pt);
\fill[red] ({50*cos(60)},{50*sin(60)}) circle (2pt);
\fill[red] ({50*cos(180)},{50*sin(180)}) circle (2pt);
\fill[red] ({50*cos(300)},{50*sin(300)}) circle (2pt);

\end{tikzpicture}
\caption{Four spheres in $\RR^2$ forming a comatching.}
\end{figure}
\end{proof}

We will rederive Theorem \ref{spheres} as a consequence of Theorem \ref{main'} in Section 3.

\subsection{Hypersurfaces} Another consequence of Theorem \ref{main'} is an optimal colorful Helly theorem for hypersurfaces of bounded degree over any field $\mathbb{F}$, generalizing the works of Moktzin \cite{Motzkin} and Deza-Frankl \cite{DezaFrankl}. We will discuss this in Section 3, but for now we would like to include a separate argument inspired by the one of Lov\'asz from \cite{Ba82}, which only works in the case when $\mathbb{F}=\mathbb{C}$ and which may be interesting for independent reasons. We would also like to highlight that in \cite{Helly}, De Loera et al also adapted this argument of Lov\'asz in order to prove an optimal colorful version of Doignon's theorem \cite{Doignon}. 

\begin{theorem}[Colorful Helly for complex polynomials] \label{polys}
Let $D$ and $d$ be positive integers and $m:=\binom{D+d}{d}$. If $\sF_1,\,\ldots,\,\sF_m\subset\mathbb{C}[x_1,\,\ldots,\,x_d]$ are finite sets of polynomials of degree at most $D$ such that
\[
\bigcap_{f\in\sF_i}Z(f)=\emptyset
\]
for all $i\in [m]$, then there exists $\{f_i\in\sF_i\}_{i\in [m]}$ such that  $\bigcap_i Z(f_i)=\emptyset$.
\end{theorem}

We will also require the following simple lemma. 

\begin{lemma}
\label{lem:closer}
Let $V$ be a $\mathbb{C}$-vector space and $a,\,b,\,c\in V$ be distinct vectors. Let $[b,\,c]:=\{tb+(1-t)c: t\in\mathbb{C}\}$. If $\langle a-c,\,b-c\rangle\neq 0$, there exists $d\in[b,\,c]$ such that $||a-d||<||a-c||$.
\end{lemma}

\begin{proof}
Let $d(t):=tb+(1-t)c$. Then $a-d(t)=(a-c)-t(b-c)$ so
\begin{align*}
||a-d(t)||^2&=(a-d(t))(\overline{a}-\overline{d(t)})\\
&=((a-c)-t(b-c))((\overline{a-c})-\overline{t}(\overline{b-c}))\\
&=(a-c)(\overline{a-c})+t\overline{t}(b-c)(\overline{b-c})-t(\overline{a-c})(b-c)-\overline{t}(a-c)(\overline{b-c})\\
&=||a-c||^2+||t||^2||b-c||^2-2\mathrm{Re}(t(\overline{a-c})(b-c)).
\end{align*}
Hence $||a-d(t)||<||a-c||$ if and only if $||t||^2||b-c||^2<2\mathrm{Re}(t(\overline{a-c})(b-c))$. Note that $(\overline{a-c})(b-c)$ is nonzero because $\langle a-c,\,b-c\rangle\neq 0$.

If $(\overline{a-c})(b-c)$ has a real part $r\neq 0$, let $t$ be real. Then $\mathrm{Re}(t(\overline{a-c})(b-c))=rt$ so it suffices to prove that there exists $t$ such that $t^2||b-c||^2<2rt$. Taking $t$ close enough to $0$ works.

If $(\overline{a-c})(b-c)$ can be written as $ri$ for some real $r$, let $t=:si$. Then $\mathrm{Re}(t(\overline{a-c})(b-c))=-rs$ so it suffices to prove that there exists $s$ such that $s^2||b-c||^2<-2rs$. Taking $s$ close enough to $0$ works.
\end{proof}

The next ingredient is an analogue of the so-called colorful Carath\'eodory theorem \cite{Ba82} for polynomials of bounded degree.

\begin{proposition}[Colorful Carath\'eodory for polynomials] \label{caratheodory}
Let $D$ and $d$ be positive integers and $m:=\binom{D+d}{d}$. Let $\sF_1,\,\ldots,\,\sF_m\subset\mathbb{C}[x_1,\,\ldots,\,x_d]$ be finite sets of polynomials of degree at most $D$. For any
\[
p\in\bigcap_{i=1}^m I(\sF_i),
\]
there exists $\{f_i\in\sF_i\}_{i\in [m]}$ such that $p\in I(f_1,\,\ldots,\,f_m)$.
\end{proposition}

\begin{proof}
Let $R:=\mathbb{C}[x_1,\,\ldots,\,x_d]$. The ring $R$ is a $\mathbb{C}$-vector space so we have an inner product and norm. Define the distance between polynomials $P$ and $Q$ to be $||P-Q||$. This is a metric. Let the distance between a polynomial $P$ and a subset $S\subset R$ equal $\min_{Q\in S}||P-Q||$.

Suppose $f_i\in\sF_i$ for $1\leq i\leq m$ has the property that the ideal $\mathfrak{a}:=I(f_1,\,\ldots,\,f_m)$ is closest to $p$, with $h\in\mathfrak{a}$ being the closest point. Assume for the sake of contradiction that $||p-h||>0$. Note that polynomials in $R$ of degree at most $D$ have at most $m$ terms. Thus the $\mathbb{C}$-vector space $V$ of polynomials of degree at most $D$ has dimension $m$.

If $f_1,\,\ldots,\,f_m$ forms a basis of $V$, there exists $c_1,\,\ldots,\,c_m\in\mathbb{C}$ not all $0$ such that $c_1f_1+\cdots+c_mf_m=1$. But then $1\in\mathfrak{a}$ so $p\in R=\mathfrak{a}$, a contradiction.

Hence there must exist $c_1,\,\ldots,\,c_m\in\mathbb{C}$ not all $0$ for which $c_1f_1+\cdots+c_mf_m=0$. Without loss of generality, suppose $c_m\neq 0$; so $\mathfrak{a}=I(f_1,\,\ldots,\,f_{m-1})$.

Suppose that for all $f\in\sF_m$ and $g\in (f)$, it is true that $\langle g-h,\,p-h\rangle=0$. Since $0\in (f)$ it follows that $\langle g,\,p-h\rangle=\langle g-h,\,p-h\rangle-\langle 0-h,\,p-h\rangle=0$. Thus $(\sF_m)$ is contained in $(p-h)^\perp$. However, $\langle p,\,p-h\rangle=\langle p-h,\,p-h\rangle-\langle 0-h,\,p-h\rangle=||p-h||^2\neq 0$ gives $p\not\in (p-h)^\perp$ so $p\not\in (\sF_m)$, a contradiction.

Hence there exists $f\in\sF_m$ and $g\in (f)$ such that $\langle g-h,\,p-h\rangle\neq0$. By Lemma~\ref{lem:closer}, there exists $h_1\in [g,\,h]\subset I(f,\,h)\subset I(f_1,\,\ldots,\,f_{m-1},f)$ closer to $p$ than $h$. Thus $I(f_1,\,\ldots,\,f_{m-1},f)$ is closer to $p$ than $\mathfrak{a}$, a contradiction.
\end{proof}

\begin{proof}[Proof of Theorem~\ref{polys}]
By Hilbert's Nullstellensatz, $1\in(\sF_i)$ for all $i\in [m]$. By Proposition \ref{caratheodory}, there exists a transversal $\{f_i\in\sF_i\}_{i\in [m]}$ such that $1\in (f_1,\,\ldots,\,f_m)$. Thus $\bigcap_i Z(f_i)=\emptyset$ by Hilbert's Nullstellensatz again.
\end{proof}

We will rederive Theorem \ref{polys} as a consequence of Theorem \ref{main'} in Section 3. In doing so, we will circumvent the need for the nontrivial direction in the Hilbert Nullstellensatz and thus get an optimal colorful Helly result for polynomials of bounded degree over any ground field $\mathbb{F}$.

\bigskip

\section{Proof of Theorem \ref{main}, Proposition \ref{noetherian} and Theorem \ref{main'}}

We begin with the proof of Theorem \ref{main'}, which subsumes Theorem \ref{main}. 

\subsection{Proof of Theorem \ref{main'}}
The setup: let $\tau = \tau'(X, \sF)$, and $\sF_1,\ldots, \sF_{\tau + 1}$ be finite subfamilies of $\sF$ with an empty intersection. We make the following key observation.
    \smallskip
    
    \textit{Claim}. For each $(\tau + 1)$-tuple $C = (F_1,\ldots, F_{\tau + 1}) \in \sF_1 \times \cdots \times \sF_{\tau + 1}$ with non-empty intersection, there exists some $C' = (F'_1,\ldots, F'_{\tau + 1}) \in \sF_1 \times \cdots \times \sF_{\tau + 1}$ such that
    $$F'_1 \cap \cdots \cap F'_{\tau + 1} \subsetneq F_1 \cap \cdots \cap F_{\tau + 1}.$$
    \smallskip 
    \textit{Proof of Claim}. We first note that there must exist an $i \in [\tau + 1]$ such that $$F_i \supset \bigcap_{j \neq i} F_j.$$
    If not, we can find $x_i \in X$ that lies in $\bigcap_{j \neq i} F_j$ but not $F_i$. Setting $x_{\tau + 2}$ to be any element in the intersection $\bigcap_{j} F_j$, we obtain a comatching with intersection of size $\tau + 1$, contradiction.

    Consider any $x \in \bigcap_{j \neq i} F_j$. As the sets in $\sF_i$ have empty intersection, there exists some $F_i' \in \sF_i$ that does not contain $x$. Letting $F_j' = F_j$ for each $j \neq i$, we have
    $$F'_1 \cap \cdots \cap F'_{\tau + 1} \subsetneq \bigcap_{j \neq i} F'_j = \bigcap_{j \neq i} F_j = \bigcap_{j} F_j$$
    which proves the claim.
    
    Let the $(\tau + 1)$-tuple $C = (F_1,\ldots, F_{\tau + 1}) \in \sF_1 \times \cdots \times \sF_{\tau + 1}$ achieve the minimal non-empty intersection $F_1\cap\cdots\cap F_{\tau+1}$ with respect to the inclusion ordering; such $F_i$'s exist because the $\sF_i$'s are finite. Applying the claim, we obtain
    $C' = (F'_1,\ldots, F'_{\tau + 1}) \in \sF_1 \times \cdots \times \sF_{\tau + 1}$ such that
    $$F'_1 \cap \cdots \cap F'_{\tau + 1} \subsetneq F_1 \cap \cdots \cap F_{\tau + 1}.$$
    By minimality, we must have $F'_1 \cap \cdots \cap F'_{\tau + 1} = \emptyset$, as desired. This proves that
    $$\eta(X, \sF) \leq 1 + \tau'(X, \sF).$$

\smallskip

In particular, if $\tau'(X, \sF) = \tau(X, \sF) - 1$, then $\eta(X, \sF) \leq \tau(X, \sF)$. On the other hand, the reverse inequality $\eta(X, \sF) \geq \tau(X,\sF)$ also holds. Write instead $\tau = \tau(X, \sF)$. Let $F_1,\ldots, F_\tau$ be a maximum comatching in $\sF$. Since $\tau'(X, \sF) = \tau - 1$, we have $F_1 \cap \cdots \cap F_\tau = \emptyset$; taking $\sF_1 = \cdots = \sF_{\tau - 1} = \{F_1,\ldots, F_\tau\}$ shows that $\sF$ has colorful Helly number at least $\tau$. All in all, we showed that
\begin{equation} \label{colorfulHelly}
     \text{if}\ \tau'(X, \sF) = \tau(X, \sF) - 1,\ \text{then}\ \eta(X, \sF)  = \tau(X, \sF).
\end{equation}

\smallskip

\subsection{Examples showing sharpness} We now show that Theorem~\ref{main} is tight by itself. For example, let $M = 2$, $X = \{1,\,2,\,3,\,4\}$, and
$$\sF_{1} = \{\{1,\,2\},\, \{3,\, 4\}\},$$
$$\sF_{2} = \{\{2,\,3\},\, \{4,\, 1\}\},$$
$$\sF = \sF_1 \cup \sF_2.$$
Then there is no comatching in $\sF$ of size $3$: any three element subset of $X$ must be of the form $\{i,\, i + 1,\, i + 2\}$ modulo $4$, and there is no subset of $\sF$ that contains $\{i, i + 2\}$. So $\tau(X, \sF) = 2$, but $\eta(X, \sF) \geq 3$ because one can consider the aforementioned two families $\sF_1$ and $\sF_2$ (which clearly do not satisfy the colorful Helly property).


A more general example goes as follows. Take any $M \geq 2$. Let $X = \{1,\,\ldots,\, 2M\}$, and 
$$\sF_{1} = \cdots = \sF_{M - 1} = \{X \backslash \{1,\,2\},\, X \backslash \{3,\, 4\},\,\ldots,\, X \backslash \{2M - 1,\, 2M\}\},$$
$$\sF_{M} = \{X \backslash \{2,\,3\},\, X \backslash \{4,\,5\},\,\ldots,\, X \backslash \{2M,\, 1\}\},$$
$$\sF = \bigcup_{i = 1}^M \sF_i.$$
Then there is no comatching of size $(M + 1)$ in $\sF$: any $(M + 1)$ element subset $M$ of $X$ must contain three elements of the form $\{i,\, i + 1,\, i + 2\}$ modulo $M$, and there is no set in $\sF$ that contains $i$ and $i + 2$ but not $i + 1$. So $\tau(X, \sF) \leq M$. On the other hand, we have $\eta(X, \sF) \geq M + 1$ because one can consider the aforementioned families $\sF_1, \cdots, \sF_M$, which do not satisfy the colorful Helly property.
\subsection{Proof of Proposition \ref{noetherian}} Let $\eta = \eta(X, \sF)$. It is clear that $\eta^\infty(X, \sF) \geq \eta$. On the other hand, consider (possibly infinite) subfamilies $\sF_1,\,\ldots,\, \sF_{\eta} \subset \sF$ such that $\cap_{F \in \sF_{i}} F = \emptyset$ for every $i=1,\,\ldots,\,\eta$. By the Noetherian hypothesis, there exists finite $\sF_i' \subset \sF_i$ such that $\cap_{F \in \sF_{i}'} F = \emptyset$. By the definition of $\eta$, there exists $F_i \in \sF_i'$ such that $F_1 \cap \cdots \cap F_{\eta} = \emptyset$. We conclude that $\eta^\infty(X, \sF) \leq \eta.$

\subsection{Applications}

Finally, we would like to record three special cases of Theorem \ref{main'}, in particular the criterion for colorful Helly from \eqref{colorfulHelly}. The first one is the fact that the colorful Helly numbers of spheres in $\mathbb{R}^{d}$ equals the Helly number $d+2$, already proved in Theorem \ref{spheres}. 

\bigskip

\noindent\textbf{Spheres}. Let $X = \R^d$, and let $\mathcal{S}_d$ be the family of all $d$-dimensional spheres. Then we have
$$\tau(X, \mathcal{S}_d) = d + 2, \tau'(X, \mathcal{S}_d) = d + 1.$$
We have shown the first statement above, which implies that $\tau'(X, \mathcal{S}_d) \in \{d + 1, d + 2\}$. To show that $\tau'(X, \mathcal{S}_d) = d + 1$, suppose for the sake of contradiction that points $x_1,\ldots, x_{d + 3} \in \R^d$ and spheres $S_1,\ldots, S_{d + 2} \in \mathcal{S}_d$ form a comatching with intersection. We apply inversion $\phi$ with center $O = x_{d + 3}$. Then $\phi(S_1),\,\ldots,\, \phi(S_{d + 2})$ are a set of $(d + 2)$-planes in $\R^d$ that form a comatching with $\phi(x_1),\ldots, \phi(x_{d + 2})$. This is impossible because $d+2$ points in $\mathbb{R}^{d}$ are always affinely dependent. So we recover the fact that
$$\eta(X, \mathcal{S}_d) = \eta^\infty(X, \mathcal{S}_d) = d + 2.$$

\bigskip

\noindent\textbf{Hypersurfaces of bounded degree}. Let $\F$ be an arbitrary field, let $X := \F^d$, and let $V_D := F[x_1,\ldots, x_d]_{\leq D}$ be the vector space of polynomials in $x_1,\ldots,x_{d}$ and of degree at most $D$. Recall that $\tau(X,\, V_D) \leq \binom{D + d}{d}$, since the elements of any comatching must correspond to linearly independent polynomials via \eqref{matching}.  Suppose for the sake of contradiction that $\tau'(X,\,V_D) \geq \binom{D + d}{d}$. Set $m = \binom{D + d}{d}$. Then there exists polynomials $f_1,\,\ldots,\,f_{m}\in F[x_1,\ldots, x_d]_{\leq D}$ and $a_0,\,a_1,\,\ldots,\,a_m\in F$ such that $f_i(a_j)=0$ if and only if $i\neq j$. Since the polynomials are linearly independent, they form a basis. Hence there exists $c_1,\,\ldots,\,c_m\in F$ such that $c_1f_1+\cdots+c_mf_m=1$. It follows that $0=c_1f_1(a_0)+\cdots+c_mf_m(a_0)=1$, a contradiction. Thus $\tau'(X,\,V_D)\leq\binom{D + d}{d}-1$. By \eqref{colorfulHelly}, it follows that 
$$\eta(X, V_D) \leq \binom{D + d}{d}.$$
In other words, this shows that Theorem \ref{polys} holds over any ground field $\mathbb{F}$, not only over $\mathbb{C}$. 

\bigskip

\noindent\textbf{Hamming balls of a fixed radius}. The Hamming distance between $p,\,q\in\mathbb{R}^n$, denoted $\mathrm{dist}(p,\,q)$, is the number of coordinates where $p$ and $q$ differ. The Hamming ball centered at $x$ with radius $t$, denoted $B(x,\,t)$, is the set of points $y\in\mathbb{R}^n$ with $\mathrm{dist}(x,\,y)\leq t$. In \cite{AJS}, Alon, Jin and Sudakov prove that the Helly number for the family $\sF$ of Hamming balls of radius $t$ in $\mathbb{R}^{n}$ is $2^{t+1}$ by establishing $\tau(X, \sF) = 2^{t + 1}$. We generalize this result by showing that the colorful Helly number of $\sF$ is also $2^{t+1}$. 

Recall that $\tau'=\tau'(X, \sF)$ is the largest $\tau'$ such that there exist $x_1,\ldots, x_{\tau' + 1} \in X$ and $F_1,\ldots, F_{\tau'}$ such that $x_i \in F_j$ if and only if $i \neq j$. We prove that $\tau'$ for Hamming balls of radius $t$ is $2^{t+1}-1$. Let $m:=2^{t+1}$. Suppose for the sake of contradiction there exists $a_1,\,\ldots,\,a_m,\,b_0,\,b_1,\,\ldots,\,b_m\in\mathbb{R}^n$ such that $\mathrm{dist}(a_i,\,b_j)\leq t$ if and only if $i\neq j$. For each $i\in [m]$, let $D_i:=\{k\in [n]: a_{i,k}\neq b_{i,k}\}$ and $d_i$ be the largest element of $D_i$. For all $i\in [m]$, let $s_i:=\mathrm{dist}(a_i,\,b_i)-t$. We call a pair $(I_1,\,I_2)$ compatible with $i$ if $I_1\subset D_i$, $|I_1|\geq t+\frac{s_1+1}{2}$, $I_2\subset [n]\setminus D_i$, or $I_1\subset D_i\setminus\{d_i\}$, $|I_1|=t+\frac{s_i}{2}$, $I_2\subset [n]\setminus D_i$. Note that $|I_1|\geq t+1$ in both cases. For every $i\in [m]$ and every such pair $(I_1,\,I_2)$, define a polynomial on $x\in\mathbb{R}^n$ by
\[
f_{i,I_1,I_2}(x):=\prod_{k\in I_1\cup I_2}(x_k-a_{i,k})\prod_{k\in D_i\setminus I_1}(x_k-b_{i,k}).
\]
In \cite[Theorem 2.3]{AJS}, Alon, Jin and Sudakov prove that these polynomials are linearly independent in the vector space of multilinear polynomials on $n$ variables $V_n = \RR[x_1, \cdots, x_n] / (x_1^2 - x_1, \cdots, x_n^2 - x_n)$. The vector space $V_n$ has dimension $2^n$. Let
\[
V_{n,d}:=\begin{cases}
\sum_{i=0}^{(d-1)/2}\binom{n}{i}, & \text{if $d$ is odd}\\
\sum_{i=0}^{d/2-1}\binom{n}{i}+\binom{n-1}{d/2-1}, & \text{if $d$ is even}
\end{cases}
\]
The number of pairs compatible with $i$ is $2^{n-(t+s_i)}V_{t+s_i,s_i}$ so
\[
\sum_{i=1}^m2^{n-(t+s_i)}V_{t+s_i,s_i}\leq 2^n.
\]
By \cite[Claim 2.2]{AJS}, $V_{t+s_i,s_i}\geq 2^{s_i-1}$ so
\[
2^n\geq\sum_{i=1}^m2^{n-(t+s_i)}V_{t+s_i,s_i}\geq\sum_{i=1}^m 2^{n-t-1}=2^n.
\]
Hence the polynomials $f_{i,I_1,I_2}$ form a basis of $V_n$, and there exists coefficients $c_{i, I_1, I_2} \in\mathbb{R}$ such that

\[
\sum_{i, I_1, I_2} c_{i, I_1, I_2}f_{i, I_1, I_2}(x) =1, \text{for all}\ x \in \{0, 1\}^{[n]}.
\]

For all $(i,\,I_1,\,I_2)$, as $b_0$ has Hamming distance at most $t$ with $a_i$, and $\abs{I_1 \cup I_2} \geq t + 1$, there must exists $k \in I_1 \cup I_2$ such that $b_{0, k} = a_{i, k}$. Therefore, the first product in the definition of $f_{i,I_1,I_2}$ vanishes at $b_0$, thus $f_{i,I_1,I_2}(b_0)=0$. Therefore, we have

\[
0=\sum_{i, I_1, I_2} c_{i, I_1, I_2}f_{i, I_1, I_2}(b_0) =1,
\]
a contradiction. Thus $\tau'=2^{t+1}-1$ and the conclusion follows by the criterion from \eqref{colorfulHelly}.

\bigskip

\section{Collapsibility}

Motivated by our results on the colorful Helly theorem, we further explore the connection between the comatching number and the collapsible and Leray properties of the nerve complex. We first recall some relevant definitions (See e.g. \cite{Khm}). All simplicial complexes in this section are abstract simplicial complexes. A \emph{$d$-collapse} of a simplicial complex removes a free face $\sigma$ with $\abs{\sigma} = d$ and all faces containing $\sigma$. We say a simplicial complex $K$ is \emph{$d$-collapsible} if we can remove all simplices of size at least $d$ via a sequence of $d$-collapses. We say $K$ is \emph{$d$-Leray} if {\color{black} for any induced subcomplex $L$ of $K$, the reduced homology groups $\Tilde{H}_i(L)$ with coefficients in $\RR$ are trivial for all $i \geq d$.}

The nerve complex $N(\sF)$ of a set family $\sF$ is the simplicial complex on vertex set $\sF$ whose faces are the subfamilies of $\sF$ with non-empty intersection. We say a complex $K$ is \emph{$d$-representable} if it is isomorphic to the nerve complex of convex sets in $\RR^d$. The connections between $d$-representable, $d$-collapsible, $d$-Leray and the Helly theorems are well-studied. Wegner \cite{Wegner} showed that $d$-representable implies $d$-collapsible, and $d$-collapsible implies $d$-Leray. In the opposite direction, Tancer \cite{Tancer} showed that for any $d$ there exists a simplicial complex that is $2$-collapsible but not $d$-representable. Alon and Kalai \cite{AK} showed that if $N(\sF)$ is $d$-collapsible, then $\sF$ has fractional Helly number $(d + 1)$ with optimal dependence between $\alpha$ and $\beta$. Furthermore, Kalai and Meshulam \cite{KalaiMeshulam} also showed that a $d$-collapsible simplicial complex property has the colorful Helly property, thereby generalizing the colorful Helly theorem for convex set (or more precisely the stronger fact that the $d$-Leray property implies that the colorful Helly number at most $d + 1$). For the sake of completeness, we include a version of their proof in the abstract setting for set systems. 
\begin{proposition} \label{Khmelnitsky}
If $N(\sF)$ is $d$-collapsible, then $\sF$ has colorful Helly number at most $(d + 1)$.
\end{proposition}
\begin{proof}
    Suppose $\sF_1,\ldots, \sF_{d + 1}$ are subfamilies of $\sF$ such that for every $F_1 \in \sF_1,\ldots, F_{d + 1} \in \sF_{d + 1}$, we have $F_1 \cap \cdots \cap F_{d + 1} \neq \emptyset$. Then $(F_1,\ldots, F_{d + 1})$ is a face in $N(\sF)$. Without loss of generality, assume $(F_1,\ldots, F_{d + 1})$ is the first face in $\sF_1 \times \cdots \times\sF_{d + 1}$ to be removed during the $d$-collapsing process by collapsing $\sigma = (F_1,\ldots, F_d)$. Let $\tau$ be the unique maximal face containing $\sigma$ right before this $d$-collapse. For any $F \in \sF_{d + 1}$, $\sigma \cup \{F\}$ is a face right before the collapse, so we must have $\sigma \cup \{F\} \in \tau$. We conclude that $\sF_{d + 1} \subset \tau$, thus $\sF_{d + 1}$ is a face in $N(\sF)$ so
    \[
    \bigcap_{F\in\sF_{d+1}}F\neq\emptyset.
    \]
\end{proof}
A similar argument can also be found in \cite{Khm}. In light of our Theorem \ref{main'}, it is natural to wonder if $\tau(X, \sF) \leq d$ implies $d$-collapsibility. If true, Proposition \ref{Khmelnitsky} would thus imply our Theorem \ref{main'} and would also yield Corollary \ref{frac} with optimal dependence between $\alpha$ and $\beta$ (via the work of Alon and Kalai \cite{AK}). Unfortunately, the following proposition shows that this is not the case.
\begin{proposition}
\label{prop:no-leray}
For every $d \geq 1$, there exist a base set $X$ and a set family $\sF$ on $X$ with comatching number $\tau'(X, \sF) \leq \tau(X, \sF) \leq 2d$, such that $N(\sF)$ is not $(3d - 1)$-Leray and thus not $(3d - 1)$-collapsible.
\end{proposition}
We now prove this proposition. First, we define an analogue of comatching number for abstract simplicial complexes.
\begin{definition}
    Let $K$ be an abstract simplicial complex on vertex set $V$. A \textit{comatching} in $K$ is a subset $M \subset V$ such that the following holds: for any $v \in M$, there is a maximal face $F$ of $K$ with $F \cap M = M \backslash \{v\}$. Let $\tau(V, K)$ denote the maximum size of a comatching in $K$.
\end{definition}
We can easily relate the definition of comatching for set families and simplicial complexes.
\begin{proposition}
\label{prop:comatching-conversion}
    1) For any set family $\sF$ on ground set $X$, we have $\tau(\sF, N(\sF)) \leq \tau(X, \sF)$.

    2) For any simplicial complex $K$ on vertex set $V$ without an isolated vertex, there is a ground set $X$ and a set family $\sF(K)$ on $X$ such that $K \cong N(\sF(K))$ and $\tau(X, \sF(K)) \leq \max(2, \tau(V, K))$.
\end{proposition}
\begin{proof}
    1) Suppose $M$ is a comatching in $N(\sF)$. We claim that $M$ is also a comatching in $\sF$. Indeed, for any set $f \in M$, there is a maximal face $F$ of $N(\sF)$ such that $F \cap M = M \backslash \{f\}$. There must exist some $a \in X$ such that
    $$F = \{f \in \sF: a \in f\}.$$
    Therefore, $a$ lies in every set in $M \backslash \{f\}$, but not in $f$. So $M$ forms a comatching in $\sF$.

    2) Let $F$ be the set of maximal faces in $K$. Let $X = V \sqcup F$. For each $v \in V$, we associate the subset $F_v = \{v\} \cup \{f \in F: v \in f\}$ on $X$. Let $\sF(K)$ be the set family $\{F_v: v \in V\}$ on $X$. 
    
    We check that the map $\phi: s \mapsto F_s$ gives an isomorphism of complexes $K \cong N(\sF(K))$. If $f_0$ is a face in $K$, then $f_0$ is contained in some maximal face $f \in F$. Therefore, for any $s \in F_0$, we have $f \in \phi(s)$, so $f \in \bigcap_{s \in F_0} \phi(s)$, thus $\phi(f_0)$ is a face in the nerve complex by definition. On the other hand, suppose $\phi(f_0)$ is a face in $K$ for some $f_0 \subset V$. If $\abs{f_0} = 1$, then $f_0$ is a face in $K$ as $K$ has no isolated vertex. If $\abs{f_0} \geq 2$, there exists some $f \in V \sqcup F$ such that $f \in \bigcap_{s \in f_0} \phi(s)$. We must have $f \in T$, thus $s \in F$ for any $s \in f_0$. So $f_0 \subset F$, thus $f_0$ is a face in $V$, as desired.
    
    Furthermore, suppose $M$ is a comatching of size at least $3$ in $\sF(K)$. We check that $\phi^{-1}(M)$ is a comatching in $K$. Consider any $v \in \phi^{-1}(M)$. As $M$ is a comatching in $\sF(K)$, there is some $f \in V \sqcup F$ such that for each $F_u \in M$, $F_u$ contains $f$ if and only if $u \neq v$. As $\abs{M \backslash \{v\}} \geq 2$, at least two sets in $\sF(K)$ contains $F$, so we must have $f \in F$. Thus for $u \in \phi^{-1}(M)$, we have $u \in f \Leftrightarrow f \in F_u \Leftrightarrow u \neq v$. We conclude that $f \cap \phi^{-1}(M) = \phi^{-1}(M) \backslash \{v\}$, and $\phi^{-1}(M)$ is a comatching in $\sF(K)$, as desired.
\end{proof}
In light of part 2) of this proposition, we can restate Proposition \ref{prop:no-leray} in terms of the comatching number for simplicial complexes (note that isolated vertices don't affect the collapsible or Leray property).
\begin{proposition}
\label{prop:no-leray-II}
For every $d \geq 2$, there exists a vertex set $V$ and a simplicial complex $K$ on $V$, such that $\tau(V, K) \leq 2d$ and $K$ is not $(3d - 1)$-Leray thus also not $(3d - 1)$-collapsible.
\end{proposition}
We first verify this proposition for $d = 1$ with an explicit construction. We make use of Leray's celebrated nerve theorem, which allows us to easily compute the homology of certain nerve complexes.
\begin{theorem}[\cite{Leray}]
\label{thm:nerve}
Let $X$ be a topological space, and let $\sF = \{U_i\}$ be an open cover of $X$. Suppose all finite, non-empty intersections of sets in $\sF$ are contractible. Then the nerve complex $N(\sF)$ is homotopy-equivalent to $X$. Therefore, $\tH_i(N(\sF)) \cong \tH_i(X)$ for every $i \geq 0$.
\end{theorem}
\begin{proposition}
    \label{prop:d=2}
    There exists a vertex set $V$ and a simplicial complex $K$ on $V$ such that $\tau(V, K) = 2$, but $K$ is not $2$-Leray.
\end{proposition}
\begin{proof}
\begin{figure}

\begin{tikzpicture}[scale=1.0, font=\small]
\tikzset{every node/.style={}}
\foreach \x in {0,...,4} {
    \draw (\x,0) -- (\x,4);
}
\foreach \y in {0,...,4} {
    \draw (0,\y) -- (4,\y);
}


\newcounter{num}
\setcounter{num}{1}

\foreach \row in {3.5,2.5,1.5,0.5} {
    \foreach \col in {0.5,1.5,2.5,3.5} {
        \node [rectangle] at (\col,\row) {\arabic{num}};
        \stepcounter{num}
    }
}
\draw[red, thick] (0.05,3.95) rectangle (1.95,2.05);

\draw[blue, thick] (0.1,3.9) rectangle (1.9,3.1);

\draw[blue, thick] (0.1,0.1) rectangle (1.9,0.9);
\end{tikzpicture}

\caption{The regions in red and blue are members of $\sF$. To visualize $K = N(\sF)$, we identify the red region with grid square $1$, and the blue region with grid square $13$.}
\label{fig:example}
\end{figure}
Let $X$ be the surface of the torus $(\RR / \ZZ)^2$. We partition $X$ into a $4 \times 4$ square grid with the sides glued together. For the set of $2 \times 2$ subsquare $S$ of the grid, let $S^\circ$ be its interior. Let $\sF$ be the family of $S^\circ$'s. Define $K$ to be the nerve complex $N(\sF)$ of $\sF$. See Figure~\ref{fig:example} above for an illustration.

We first check that $\tau(\sF, K) = 2$. We identify each $S^\circ \in \sF$ with its top-left grid square. Under this identification, the maximal faces of $K$ are the $2 \times 2$ subsquares of the grid. Suppose we have a comatching $M$ of size $3$ in $K$. Then each pair of corresponding grid squares in $M$ must lie in a common maximal face, which is possible only if the corresponding grid squares in $M$ form a right isoceles triangle. Without loss of generality, we may assume that $M$ consists of grid squares $\{1,2,5\}$. Then there is no maximal face that contains $\{2, 5\}$ but not $1$, contradiction.

We now compute the homology of $K$. As the non-empty intersections of subfamilies of $\sF$ are  contractible, Leray's nerve theorem shows that $K$ is homotopic-equivalent to $X$. Therefore, we have $\tH_2(K) \cong \tH_2(X) \cong \tH_2((\RR / \ZZ)^2) \cong \RR$, so $K$ is not $2$-Leray. 
\end{proof}




We generalize the construction to arbitrary comatching number by using the join \cite[Definition 2.16]{Koz} of two simplicial complexes. 
\begin{definition}
    Let $K$ be a simplicial complex on $V$ and $L$ be a simplicial complex on $W$. The \emph{join} of $K$ and $L$, denoted by $K*L$, is the simplicial complex on $V \sqcup W$ with a face $F \sqcup G$ for each face $F$ in $K$ and face $G$ in $L$.
\end{definition}
We now illustrate how the comatching number and the homology groups behave under joins.
\begin{proposition}
    \label{prop:join-comatching}
    We have $\tau(V \sqcup W, K*L) \leq \tau(V, K) + \tau(W, L)$.
\end{proposition}
\begin{proof}
    If $P$ is a comatching in $K*L$, then $P \cap V$ is a comatching in $K$ and $P \cap W$ is a comatching in $L$. So $\abs{P} = \abs{P \cap V} + \abs{P \cap W} \leq \tau(K) + \tau(L)$.
\end{proof}
\begin{definition}
    We say $K$ is $d$-good if $\tH_d(K) \neq 0$ and $\tH_i(K) = 0$ for every $i > d$.
\end{definition}
\begin{proposition}
    \label{prop:join-homology}
    If $K$ is $d$-good and $L$ is $d'$-good, with $d, d' \geq 1$, then $K * L$ is $(d + d' + 1)$-good.
\end{proposition}

\begin{proof}
    Fix any $k \geq d + d' + 1$.  By an analog of K\"{u}nneth's theorem for join (see e.g. \cite[Section 3]{KalaiMeshulam}, \cite{Whitehead}), we have
    $$\tH_{k}(K * L) \cong \bigoplus_{i + j = k - 1} \tH_{i}(K) \otimes \tH_j(L).$$
    By the goodness assumption on $K$ and $L$, all terms in the direct sum vanish except when $k = d + d' + 1$, $i = d$ and $j = d'$. So $\tH_{k}(K * L)$ is nonzero when $k = d + d' + 1$ and zero when $k > d + d' + 1$.
\end{proof}
\begin{proof}[Proof of Proposition~\ref{prop:no-leray-II}]
    In Proposition~\ref{prop:d=2}, we constructed a simplicial complex $K_2$ such that $\tau(K_2) = 2$ and $K_2$ is $2$-good. Now we take $K_{2d} = K_2 * \cdots * K_2$, where there are $d$ copies of $K_2$. By Proposition~\ref{prop:join-comatching}, we have 
    $$\tau(K_{2d}) \leq d\tau(K_2) \leq 2d.$$ 
    By Proposition~\ref{prop:join-homology}, $K_{2d}$ is $(3d - 1)$-good, thus not $(3d - 1)$-Leray. So $K_{2d}$ satisfies the desired properties.
\end{proof}

We find it interesting to see whether bounded comatching number implies $d$-collapsible or $d$-Leray for some $d$. We already don't know the answer to the following question.
\begin{question}
\label{conj:no-leray?}
   Is there an absolute constant $d  > 0$ such that the following holds: for every set family $\sF$ with $\tau(\sF) \leq 2$, the nerve complex $N(\sF)$ is always $d$-Leray?
\end{question}

\end{document}